\documentclass[a4paper,egregdoesnotlikesansseriftitles,final]{scrartcl}
\usepackage[utf8]{inputenc}
\usepackage[british]{babel}
\usepackage[T1]{fontenc}
\usepackage{lmodern}
\usepackage{microtype}
\usepackage{amsfonts,amsthm,amssymb,amsmath}
\usepackage{mathtools}
\usepackage{letterswitharrows}
\usepackage{xcolor}
\usepackage{graphicx}
\usepackage{tikz}
\usetikzlibrary{arrows}
\usepackage[colorlinks,unicode,bookmarks,pdfusetitle,linkcolor={red!60!black},citecolor={green!60!black},urlcolor={blue!60!black},bookmarksnumbered,bookmarksopen]{hyperref}
\usepackage[capitalize]{cleveref}
\crefname{enumi}{}{}
\crefname{property}{property}{properties}
\crefname{LEM}{Lemma}{the Lemmas}
\usepackage{enumitem}
\usepackage{thmtools, thm-restate} %
\usepackage{cite}
\usepackage[abbrev]{amsrefs}
\usepackage{doi}
\renewcommand{\PrintDOI}[1]{\doi{#1}}

\newtheorem{MAINTHM}{Theorem}
\newtheorem{UNRAVEL}[MAINTHM]{Problem}
\crefname{UNRAVELREF}{Problem}{Problems}
\newtheorem{THM}{Theorem}[section]
\newtheorem{LEM}[THM]{Lemma}
\newtheorem{COR}[THM]{Corollary}

\newtheorem{PROP}[THM]{Proposition}

\theoremstyle{definition}

\newtheorem{PROB}[THM]{Problem}
\newtheorem{PROPERTY}[THM]{Property}
\newcommand{\abs}[1]{\lvert#1\rvert}

\newcommand{\menge}[1]{\left\{#1\right\}}

\renewcommand{\phi}{\varphi}

\newcommand{\N}{\mathbb{N}}
\newcommand{\R}{\mathbb{R}}

\newcommand{\join}{\lor}
\newcommand{\meet}{\land}
\newcommand{\sub}{\subseteq}
\newcommand{\sm}{\smallsetminus}

\newcommand{\cX}{\mathcal{X}}

\title{The Unravelling Problem}
\author{Christian Elbracht \and Jakob Kneip \and Maximilian Teegen}
\begin{document}
\maketitle
\begin{abstract}
We identify and study a simple combinatorial problem that is derived from submodularity issues encountered in the theory of tangles of graphs and abstract separation systems.
\end{abstract}

\section{Introduction}
Here is an intriguingly simple combinatorial problem -- simple enough that you can explain it to a first-year student of mathematics -- but which is challenging nonetheless:
\begin{UNRAVEL}[Unravelling problem]
    \label[UNRAVELREF]{prop:fickle}
    A finite set $\cX$ of finite sets is \emph{woven} if, for all $X,Y\in\cX$, at least one of $X \cup Y$ and $X\cap Y$ is in $\cX$.
    Let $\cX$ be a non-empty woven set. Does there exist an $X \in \cX$ for which $\cX - X$ is again woven?
\end{UNRAVEL}
Given a set of subsets $\cX$ which is woven, an \emph{unravelling} of $\cX$ shall be a sequence $\cX=\cX_n\supseteq \dots\supseteq \cX_0=\emptyset$ of sets, all woven, such that $\abs{\cX_i\sm \cX_{i-1}}=1$ for all $1\le i\le n$. If the unravelling problem has a general positive answer, then every woven set will have an unravelling.

The question of whether every woven set has an unravelling arose naturally in our study of structurally submodular separation systems. These systems are a framework developed by Diestel, Erde and Weißauer \cite{AbstractTangles} to generalise the theory of tangles in graphs to an abstract setting, allowing the application of tangles in a multitude of contexts including graphs and matroids, but also other combinatorial structures.

In this paper we analyse the unravelling problem. We prove affirmative versions in two important cases, which come from the original context where the problem arose: submodular separation systems. Our first main result is that unravellings exist for sets $\cX$ that consist, for some submodular function $f$ on the subsets of $V=\bigcup \cX$, precisely of the sets $X\subseteq V$ with $f(X)<k$ for some integer $k$. Our second main result settles the unravelling problem for general finite posets, which we call \emph{woven} if they contain, for every two elements, either an infimum or a supremum of these two elements. %

We start in \cref{sec:prelim} with the definitions required for this paper and show that the unravelling problem has an equivalent formulation in terms of distributive lattices. We also establish a kind of converse of unravelling, showing that we can find, for every woven set $\cX$, some subset of $\bigcup \cX$ which we can add to $\cX$ and remain woven. 

In \cref{sec:history} we give a brief overview of the theory of abstract separation systems, the setting for general tangle theory, and explain how a version of the unravelling problem naturally arises in that context.
Throughout the paper, we will come back to the context of abstract separation systems, to discuss how our results apply there.
Apart from these discussions, however, our results are independent of \cref{sec:history}, so that the reader may opt to skip it and head directly to \cref{sec:unravelling}. 

There  we give a partial solution to the unravelling problem, by showing that the following class of woven sets, which is particularly important in the theory of tangles, can indeed be unravelled:
Let $\cX$ be a collection of subsets of some finite set $V$. If $\cX$ has the form $\cX=\{X\in 2^V\mid f(X)<k\}$ for some function $f:2^V\to \R$ and $k\in \R$, let us say that $f$ \emph{induces} $\cX$.

\begin{restatable}{MAINTHM}{cororderfun} If $\cX\subseteq 2^V$ is induced by a submodular function on $2^V$, then $\cX$ can be unravelled.
\end{restatable}

In \cref{sec:sandwich} we introduce a possible generalisation of the unravelling problem from subsets of some power set to subsets of any lattice.
We show that the lattice analogue of the unravelling problem can be answered in the negative for non-distributive lattices, constructing an explicit counterexample. However, if we restrict this generalised formulation of the unravelling problem to distributive lattices, it becomes equivalent to \cref{prop:fickle}. %

We conclude in \cref{sec:weak_unravel} with our second main result, a variant of the unravelling problem for general partially ordered sets. Let us call a finite poset $P$ \emph{woven} if there exists, for any $p,q\in P$, either a supremum or an infimium in $P$. A sequence $P=P_n\supseteq \dots \supseteq P_0=\emptyset$ of subposets is an \emph{unravelling} of $P$ if $P_i$ is woven and $\abs{P_i\sm P_{i-1}}=1$ for every $1\le i \le n$. Our second main result is that all woven posets have an unravelling:
\begin{restatable}{MAINTHM}{thmselfwoven}
Every woven poset can be unravelled.
\end{restatable}
Wovenness in posets corresponds to the most general notion of structural submodularity for separation systems, which we also discuss in \cite{Structuresubmodular}. There we also present 
results about submodularity in separation systems which arose in our research on the unravelling problem.

\section{Preliminaries}\label{sec:prelim}
We will formulate a problem equivalent to \cref{prop:fickle} in terms of lattices.
This problem might be easier to work with, and will also allow us to explain how unravelling problem originally came up. But more on that in \cref{sec:history}. Let us begin by recalling some basic terminology from lattice theory.
We largely follow the notation of \cite{LatticeBook}, however we will tacitly assume that all the sets we consider are finite.

A \emph{lattice} is a non-empty partially ordered set (or `poset') $L$ in which any two elements $a,b\in L$ have a supremum and an infimum, that is, there is a unique element $a\join b$ (their \emph{join} or \emph{supremum}) minimal such that $a\le a\join b$ and $b\le a\join b$ and a unique element $a\meet b$ (their \emph{meet} or \emph{infimum}) maximal such that $a\ge a\meet b$ and $b\ge a\meet b$.

Every (finite) lattice has a greatest (\emph{top}) element and a least (\emph{bottom}) element, that is, an element $t\in L$ with $a\le t$ for every $a\in L$ and an element $b\in L$ with $b\le a$ for every $a\in L$. 

A lattice is \emph{distributive} if it satisfies the distributive laws $a\join (b\meet c)=(a\join b)\meet (a\join c)$ and $a\meet (b\join c)=(a\meet b)\join (a\meet c)$ for all $a,b,c\in L$. 

Given a poset $P$ and $a,b\in P$ we say that $a$ is a \emph{upper (lower) cover} of $b$ if $a>b$ ($a<b$) and there does not exists any $c\in P$ such that $a>c>b$ ($a<c<b$).

The power set of any given set $V$ forms a lattice (the so-called \emph{subset lattice}) with the partial order on $2^V$ given by the subset relation.
In this lattice, the join of two subsets of $V$ is their union and the meet is their intersection.
Using this, we can formulate a lattice-theoretical variation of the unravelling problem. If $L$ is a lattice, we say that $P\subseteq L$ is \emph{woven} in $L$ if for any $p,q\in P$ we have $p\join q\in P$ or $p\meet q\in P$.

We can now generalise \cref{prop:fickle} as follows:
\begin{PROB}
    \label{prop:woven_lattice}
   Let $L$ be a finite lattice and $P\subseteq L$ a non-empty woven subset. Does there exist $p \in P$ for which $P - p$ is again woven?
\end{PROB}
For distributive lattices, \cref{prop:woven_lattice} is equivalent to \cref{prop:fickle} by Birkhoff's representation theorem (see \cite{LatticeBook}), which says that every finite distributive lattice is isomorphic to a sublattice of the subset lattice of some finite set. For general lattices, however, we have a negative solution to \cref{prop:woven_lattice}: in \cref{sec:sandwich} we shall construct a (non-distributive) counterexample for \cref{prop:woven_lattice}.

Perhaps surprisingly, it is easy to establish a kind of converse to \cref{prop:woven_lattice}: given a lattice $L$ and a woven poset $P\subseteq L$ we can always find a~$ p\in L\sm P $ which one can~\emph{add} to~$ P $ while keeping it woven.
\begin{PROP}\label{prop:ravel}
	If~$ L $ is a lattice and $ P\subsetneq L $ a proper woven subset of $L$, then there is a $p\in L\sm P$ such that~$ P+p $ is again woven.
\end{PROP}
\begin{proof}
	Let~$ p $ be a maximal element of~$ L\sm P$. Then~$ P'\coloneqq P+r $ is woven: for each~$ 
	q\in P' $ we have~$ (p\join q)\in P' $ by the maximality of~$ p $ in~$ L\sm P $.
\end{proof}
In terms of woven sets in the sense of \cref{prop:fickle}, this statement directly implies the following:
\begin{COR}\label{cor:ravel_sets}
	If~$ V $ is a finite set and $\cX\subsetneq 2^V $ woven, then there is a $X\subseteq V$ such that $X\notin \cX$ and such that $\cX+X$ is again woven.
\end{COR}

\section{Abstract separation systems and submodularity}
\label{sec:history}
In this section we briefly describe the evolution of the theory of abstract separation systems and how they relate to the unravelling problem.

This theory originates in work by Robertson and Seymour almost thirty years ago, where they introduced, as part of their graph minor project \cite{GMX} a new tool to capture indirectly what `highly connected regions' of a graph are: so called \emph{tangles}. This concept of tangles has, over the years, been generalised to many other contexts, see for example \cites{ProfilesNew,AbstractTangles,TangleTreeAbstract,TangleTreeGraphsMatroids,ProfileDuality,FiniteSplinters,TanglesInMatroids,AbstractSepSys}. 

In this process of generalisation, concrete separations of a graph were replaced by a more general objects, called an abstract separation system, which generalises not only the notion of separations of a graph, but also of other combinatorial objects, for example matroids \cite{AbstractSepSys}. Formally, an abstract separation system $\vS$ is a partially ordered set with an order-reversing involution $^\ast$.
The elments $\vs \in \vS $ are called \emph{(oriented) separations} and we write $\sv$ to mean the \emph{inverse} $\vs^*$ of $\vs$. The set $\{\vs,\sv\}$ of a separation together with its inverse is also denoted as the \emph{unoriented separation} $s$, and we write $S$ for the set of all these unoriented separations. Conversely, given a set $S$ of unoriented separations we denote as $\vS$ the set of oriented separations $\vs$ for which $s\in S$.

For simplicity, we will use definitions of oriented and unoriented separations interchangeably, as long as the meaning is clear. 

 This concept of abstract separation systems was developed to have the bare minimum of structure needed to make the two fundamental theorems of tangle theory, the tangle-tree-duality \cite{TangleTreeAbstract,AbstractTangles} and the tree-of-tangles theorem \cite{ProfilesNew,AbstractTangles} work. Proving theorems, such as these two, in this simple and abstract setting is a way to obtain results in a variety of contexts, with just one unified proof.

In order for the two theorems to actually hold in this abstract setting, one needs the separation system to be `rich enough' in some sense. Historically this `richness' was ensured via so-called \emph{submodular order functions}: we require that the considered separation system is contained in some \emph{universe $\vU$ of separations}, 
a separation system $\vU$ which is a lattice. A function $f\colon \vU \to \R^+_0$ is then said to be a \emph{submodular order function} if $f$ is symmetric, i.e. $f(\vs)=f(\sv)$ for every $\vs\in \vU$ and $f$ satisfies $f(\vs)+f(\vt)\ge f(\vs\join \vt)+f(\vs\meet\vt)$ for any $\vs,\vt\in \vU$. Historically, the only separation systems considered where then the $\vS_k$'s, separation systems of the form $\vS_k:=\{\vs\in \vU\mid f(\vs)<k\}$ for some $k\in \R$.

Such a universe $\vU$ of separations, as well as such a submodular order function, naturally arise when considering separations of graphs.\footnote{The universe of separations of a graph $G$ is given by the set of all separations $(A,B)$ of $G$, and the considered order function is given by $f((A,B))=\abs{A\cap B}$}

However, in~\cite{AbstractTangles}, Diestel, Erde and Weißauer found a weaker structural condition, which still provided enough `richness' for the two main theorems to hold: the property of \emph{structural submodularity}. A separation system $\vS$ inside a universe $\vU$ of separations is said to be \emph{structurally submodular in $\vU$} if $\vs\join \vt\in \vS$ or $\vs\meet\vt\in \vS$ whenever $\vs,\vt\in \vS$. The separation $ \vr\join\vs $ and~$ \vr\meet\vs $ are then called the \emph{corners} of $\vs$ and $\vt$.
This definition no longer relies on a submodular order function and is, in many arguments, the only necessary property, as \cite{AbstractTangles} demonstrated. Moreover, given a submodular order function $f$, every set $\vS_k$ is structurally submodular in $\vU$ due to the fact that $f$ is submodular. However, as we show in \cite{Structuresubmodular}, the converse is not true, so structurally submodular separation systems are a strictly larger class of separation systems.
The most abstract versions of the two main theorems of abstact tangle theory are formulated in the context of these structural submodular separation systems \cites{AbstractTangles,FiniteSplinters,CanonicalToT,ToTviaTTD}.

Also our original motivation for considering the unravelling problem originates in these structural submodular separation systems: Given such a structural submodular separation system $\vS$ inside a universe $\vU$ of separations, it might be possible that we find a separation $\vs$ inside $\vS$ which we can delete, together with its inverse, and be left with a separation system $\vS\sm \{\vs,\sv\}$ that is again structurally submodular in $\vU$. Formally, given a structurally submodular separation system $\vS$ inside a universe $\vU$ of separations, we are interested if the following property holds:
\begin{PROPERTY}\label{assertion:unravelling} There is an $\vs\in S$ such that $ \vS\sm \{\vs,\sv\} $ is submodular in~$\vU$.
\end{PROPERTY}
If this were to hold for all structurally submodular separation systems, then we could recursively apply this reduction step to \emph{unravel} such a separation system, i.e.\ we would obtain a sequence ${\emptyset=\vS_1\subseteq \vS_2\dots\subseteq \vS_n=\vS}$ of structurally submodular separation systems such that, for every $i<n$, we have that $\vS_{i+1}\sm \vS_i$ consists of just one separation $\vs_i$ together with its inverse. Such an unravelling sequence would be of particular use for proving theorems about structurally submodular separation systems via induction. For example, it is possible to obtain a short proof of a tree-of-tangles theorem for structurally submodular separation systems via this unravelling sequence \cite{kneip2020tangles}*{Section 4.1.8}.

This question, whether \cref{assertion:unravelling} holds for every structurally submodular separation systems, is now closely related to \cref{prop:woven_lattice}. In fact, if we could unravel every structurally submodular separation system, we could answer \cref{prop:woven_lattice} positively: If there exists a woven poset $P$ inside a lattice $L$, such that $P-p$ is not woven, we could construct a structurally submodular separation system inside a universe $\vU$ of separations which can not be unravelled. We use such a construction in \cref{sec:sandwich} to turn our counterexample to \cref{prop:woven_lattice} into an example of a structurally submodular separation system inside a non-distributive lattice which cannot be unravelled.

Also, the converse of \cref{prop:woven_lattice} established in \cref{prop:ravel} directly translates to a similar statement about structurally submodular separation systems inside a universe of separations.
\begin{COR}\label{cor:ravel_uni}
	If~$ U $ is a universe of separations and $ S\subsetneq U $ submodular in $U$, then so is~$ S+r $ for some~$ r\in U\sm S $.
\end{COR}
\begin{proof} Let~$ \vr $ be a maximal element of~$ \vU\sm\vS $. By \cref{prop:ravel}, the separation system ~$ S'\coloneqq S+r $ is again submodular in $U$.
\end{proof}

 \section{Unravelling order-induced sets}%
\label{sec:unravelling}

In this section we show that for a subclass of the woven subsets of a lattice we indeed have unravellings.

For this let us say that a set $P$ inside a lattice $L$ is \emph{order-induced} if there exists a submodular function $f:L\to \R_0^+$ and a real number $k$ such that $P=\{p\in L\mid f(p)<k\}$. Here, $f$ beeing submodular should mean that $f(p)+f(q)\ge f(p\join q)+f(p\meet q)$ for any $p,q\in L$. Note that every order-induced set $P$ is woven, as the submodularity of $f$ implies that at least one of $f(p\join q),f(p\meet q)$ is at most $\max\{f(p),f(q)\}$ and thus at least one of $p\join q$ and $p\meet q$ lies in $P$, whenever both $p$ and $q$ lie in $P$.
However, there do exists woven sets which are not order-induced, see \cite{Structuresubmodular}.

We will see in what follows that for order-induced posets $P$ it is possible to find an \emph{unravelling}, that is a sequence $P=P_n\supseteq \dots \supset P_0=\emptyset$ of posets which are woven in $L$ such that $\abs{P_i\sm P_{i-1}}=1$ for every $1\le i\le n$.%

We say that $P$ can be \emph{unravelled} if there exists an unravelling for $P$.
In other words~$ P $ can be unravelled if we are able to successively delete elements from~$ P $ until we reach the empty set and maintain the property of being woven throughout.

We shall demonstrate that every order-induced subset of a lattice can be unravelled.

\begin{THM}\label{thm:unravelling}
	Let~$ L $ be a lattice with a submodular function~$ f $ and consider the subset $ P=\{p\in L\mid f(p)<k\} $ for some~$ k $. Then~$ P $ can be unravelled.
\end{THM}

For the remainder of this section let~$ L $ be a lattice with a submodular order function~$f$ and~$ P\sub L $. It is easy to see that we can perform the first step of an unravelling sequence:

\begin{LEM}\label{lem:delete-highest-order}%
	If~$ P=\{p\in L\mid f(p)<k\} $ and~$ p\in P $ maximises~$ f(p) $ in~$ S $, then~$ P-p $ is  woven in $L$.
\end{LEM}

\begin{proof}
	Given $q,r\in P-p$, since~$ P $ is woven in $L$ at least one of $q\join r$ and $q\meet r$ also lies in $P$. However, by the choice of~$ p $ we have~$ f(p)\ge f(q),f(r) $. Thus if one of $q\join r$ and $q\meet r$ equals $p$, the other also needs to lie in $P$. Thus $P-p$ is indeed woven in $L$.
\end{proof}

Unfortunately we cannot rely solely on~\cref{lem:delete-highest-order} to find an unravelling of $P$, since after its first application and the deletion of some~$ p $ the remaining poset~$ P-p $ may no longer be order-induced. This can happen if~ $P-p$ contains an $r$ such that $f(r)=f(p)$.

To rectify this, and thereby allow the repeated application of~\cref{lem:delete-highest-order}, we shall perturb the submodular function $f$ on $L$ to make it injective, whilst maintaining its submodularity and the assertion that~$ P=\{p\in L\mid f(p)<k\} $ for a suitable~$ k $.
This approach is similar to -- and inspired by -- the idea of \emph{tie-breaker functions} employed by Robertson and Seymour~\cite{GMX} to construct certain tree-decompositions.
For this we show the following:

\begin{THM}\label{thm:tiebreaker} %
    Let~$ L $ be a lattice, then there is an injective submodular function~$ {\rho\colon L\to\N} $. Moreover, we can chose $\rho$ so that, for any ${p_1,p_2,q_1,q_2\in L}$, we have that $\rho(p_1)+\rho(p_2)=\rho(q_1)+\rho(q_2)$ if and only if $\{p_1,p_2\}=\{q_1,q_2\}$.
\end{THM}

\begin{proof}
	Enumerate~$ L $ as~$ L=\menge{p_1,\dots,p_n} $. For~$ q\in L$ let~$ I(q) $ be the set of all~$ i\le n $ with~$ p_i\leq q$. We define~$ \rho\colon L\to\N $ by letting
	\[ \rho(q)=3^{n+1}-\sum_{i\in I(q)}3^i\,. \]
	To see that this function is submodular note that for~$ q $ and~$ r $ in~$ L $ we have~$ I(q)\cap I(r)=I(q\meet r) $ and~$ I(q)\cup I(r)\sub I(q\join r) $.  Therefore each~$ i\le n $ appears in~$ I(q) $ and~$ I(r) $ at most as often as it does in~$ I(q\join r) $ and~$ I(q\meet r) $. This establishes the submodularity.
	
	It remains to show that~$ \rho(q)\ne\rho(r) $ for all~$ q\ne r $. For this note that by definition of~$ \rho $ we have~$ \rho(q)=\rho(r) $ if and only if~$ I(q)=I(r)$. But if $q\neq r$ then either $q\notin I(r)$ or $r\notin I(q)$.
	
	To see the moreover part we note that $\rho(p_1)+\rho(p_2)=\rho(q_1)+\rho(q_2)$ if and only if $I(p_1)\cup I(p_2)=I(q_1)\cup I(q_2)$ and $I(p_1)\cap I(p_2)=I(q_1)\cap I(q_2)$. Since $I(p_1),I(p_2),I(q_1),I(q_2)$ correspond to the down-closures of $p_1,p_2,q_1,q_2$ in $L$, this implies that $\{p_1,p_2\}=\{q_1,q_2\}$: Clearly, if $p_1=q_1$ then we need to have $p_2=q_2$, so suppose that $\{p_1,p_2\}$ and $\{q_1,q_2\}$ are disjoint. Since $p_1\in I(p_1)$ we see that $p_1\in I(q_1)\cup I(q_2)$, so suppose without loss of generality that $p_1<q_1$. Since $q_1\in I(q_1)$ and $q_1\notin I(p_1)$ we thus conclude that $q_1\in I(p_2)$, thus $q_1<p_2$. As $p_2\in I(p_2)$ this then implies $p_2<q_2$. As $q_2\in I(q_2)$ this is a contradiction as $q_2\notin I(p_1)\cup I(p_2)$.
\end{proof}

We immediately obtain the following corollary about universes of separations:
\begin{COR}\label{cor:tiebreaker_uni}
	Let~$ U $ be a universe of separations. Then there is a submodular order function~$ \gamma\colon\vU\to\N $ with~$ \gamma(r)\ne\gamma(s) $ for all~$ r\ne s $.
\end{COR}
\begin{proof}
 Let $\rho$ be the function obtained from \cref{thm:tiebreaker} applied to $U$ as a lattice. We set $\gamma(s)=\rho(\vs)+\rho(\sv)$. It is easy to see that this is a submodular order function. The moreover part of \cref{thm:tiebreaker} guarantees that indeed $\gamma(r)\neq \gamma(s)\;\forall r\neq s$.
\end{proof}

We can now establish~\cref{thm:unravelling}.

\begin{proof}[Proof of~\cref{thm:unravelling}.]
	Let~$ L $ be a lattice with a submodular order function~$f$. Let~$ P=\{p\in L\mid f(p)<k\} $ for some~$ k\in \R_0^+$. Let~$ \rho $ be the submodular function on~$ L $ from~\cref{thm:tiebreaker}. Let $\epsilon$ be the minimal difference between two distinct values of $f$, that is $\abs{f(p)-f(q)}\ge \epsilon$ or $f(p)=f(q)$ for any two $p,q\in L$. Since $L$ is finite, $\epsilon>0$. Pick a positive constant~$ c\in\R^+ $ so that~$ c\cdot\rho(p)<\epsilon $ for all~$ p\in L $. We define a new function~$ g\colon L\to \R^+_0 $ on~$ L $ by setting
	\[ g(p)\coloneqq f(p)+c\cdot\rho(p)\,. \]
    Then~$ g $ is submodular and, like~$ \rho $, has the property that~$ g(p)\ne g(q) $ whenever~$ p\ne q $. Enumerate the elements~$ p_1,\dots,p_n $ of~$ P $ so that~$ g(p_1)<g(p_2)<\dots<g(p_n) $. Then~$P_i:= \menge{p_1,\dots,p_i}\sub P $ is woven in $L$ for each~$ i\le n $: for~$ i=n $ it equals~$ P $, and for~$ i<n $ we have that~$ P_i=\{p\in L\mid g(p)<g(p_{i+1})\} $, which is woven in $L$ since~$ g $ is a submodular function on $L$. Thus $P=P_n\supseteq \dots \supseteq P_0=\emptyset$ is an unravelling for $P$.
\end{proof}

\cref{thm:unravelling} allows us to give a class of sets $\mathcal{X}\subseteq 2^V$ for which we can answer \cref{prop:fickle} positively. We say that a function $f:2^V\to \R$ is \emph{submodular} if $f(X)+f(Y)\ge f(X\cup Y)+f(X\cap Y)$ for all $X,Y\in 2^V$ and obtain the following theorem as a corollary:
\cororderfun*
\begin{proof}
 By adding a large constant to $f(X)$ for every $X\subseteq V$ we may suppose that $f(X)\ge 0 \;\forall X\subseteq V$. Applying \cref{thm:unravelling} to the subset-lattice $2^X$ together with its subset $\mathcal{X}$ results in the desired sequence $\emptyset=\cX_0\subseteq \cX_1\subseteq \dots \subseteq \cX_n=\cX$.
\end{proof}

Moreover, \cref{thm:unravelling} also allow us to show that separation systems $\vS_k$ inside a universe of separations with a submodular order function can be unraveled.

\begin{COR}\label{cor:unravelling_uni}
	Let~$ \vU $ be a universe of separations with a submodular order function~$ f $ and~$ \vS=\vS_k $ for some~$ k $. Then~$ \vS $ can be unravelled.
\end{COR}
\begin{proof}
 Perform the same argument as in the proof of \cref{thm:unravelling}, using the function $\gamma$ from \cref{cor:tiebreaker_uni} instead of the function $\rho$ from \cref{thm:tiebreaker}.
\end{proof}

\section{A woven subset of a lattice which cannot be unravelled}\label{sec:sandwich}%

In this section we are going to construct a counterexample to \cref{prop:woven_lattice} for non-distributive lattices.
 So, we construct a lattice $L$ together with a woven subset $P$ of $L$ so that $P-p$ is not woven in $L$ for any $p\in P$. %

This construction needs to be such that for every element $p$ of $P$ there are elements $q$ and $r$ of $P$ such that either $p=q\join r$ and $q\meet r\not \in P$ or $p=q\meet r$ and $q\join r\not \in P$.

We will construct our lattice $L$ by building its Hasse diagram. To be able to prove that our construction results in a lattice we need to start with a graph of high girth. (For the definition of graphs and their basic properties we follow \cite{DiestelBook16noEE}.) Specifically we will use a 4-regular graph of high girth as a starting point. Lazebnik and Ustimenko have constructed such graphs:
\begin{LEM}[\cite{LAZEBNIK1995275}]
 There exists a $4$-regular graph $G$ with girth at least $11$.
\end{LEM}
For contradiction arguments we will try to find short closed walks in our graph. The following simple lemma then tells us that these contradict the high girth of $G$:
\begin{LEM}\label{lem:walk_cycle}
 If $G$ is a graph, $W=v_1v_2\dots v_nv_1$ a closed walk in $G$ such that there exists an $j$ with $v_i\neq v_j$ for all $i\neq j$ and $v_{j-1}\neq v_{j+1}$, then $W$ contains a cycle. In particular, $G$ contains a cycle of length at most $n$.
\end{LEM}
\begin{proof}
 Since $v_j\neq v_i$ for all $i\neq j$, the graph $W-v_j$ is connected. Thus, $W-v_j$ contains a path between $v_{j-1}$ and $v_{j+1}$ which together with $v_j$ forms the desired cycle.
\end{proof}

We are now ready to start the construction of our lattice $L$ together with its woven subset $P$.

Let $G$ be a $4$-regular graph of girth at least $11$. The ground set of our lattice $L$ consists of a top element $t$, a bottom element $b$ and $4$ disjoint copies of $V(G)$ which we call $V^-,V,W$ and $W^+$.

We say that $v\in V^-\cup V\cup W\cup W^+$ \emph{corresponds} to $w\in V^-\cup V\cup W\cup W^+$ if they are copies of the same vertex in $V(G)$.

We now start with defining our partial order on $L$. We define, for $v\in V$ and $w\in W$, that $v\le w$ if and only if there is an edge between $v$ and $w$ in $G$.

Now consider the bipartite graph $G'$ on $V\cup W$ where $v\in V$ is adjacent to $w\in W$ if and only if $v\le w$. This bipartite graph is $4$-regular graph and has girth at least $12$. Every regular bipartite graph has a 1-factor. Hence $G'$ has a colouring of $E[G']$ with two colours, \emph{red} and \emph{blue} say, such that every vertex in $G'$ is adjacent to exactly two red and exactly two blue edges. We fix one such colouring.

To define our partial order for $v^-\in V^-$ and $v\in V$ we define that $v^-\le v$ if and only if there is a red edge between $v$ and the vertex in $W$ corresponding to $v^-$. Thus, every $v^-$ in $V^-$ lies below exactly two points in $V$, we call these the \emph{neighbours in $V$} of $v^-$.

Similarly, we let $w\le w^+$ for $w\in W$ and $w^+\in W^+$ if and only if there is a blue edge between $w$ and the vertex in $V$ corresponding to $w^+$. We call the two points in $W$ which lie below $w^+ \in W^+$ the \emph{neighbours in $W$} of $w^+$.

We finish our definition of $\le$ by taking the transitive closure and defining $b\le v$ and $v\le t$ for every $v\in L$. It is easy to see that this $\le$ is indeed a partial order.

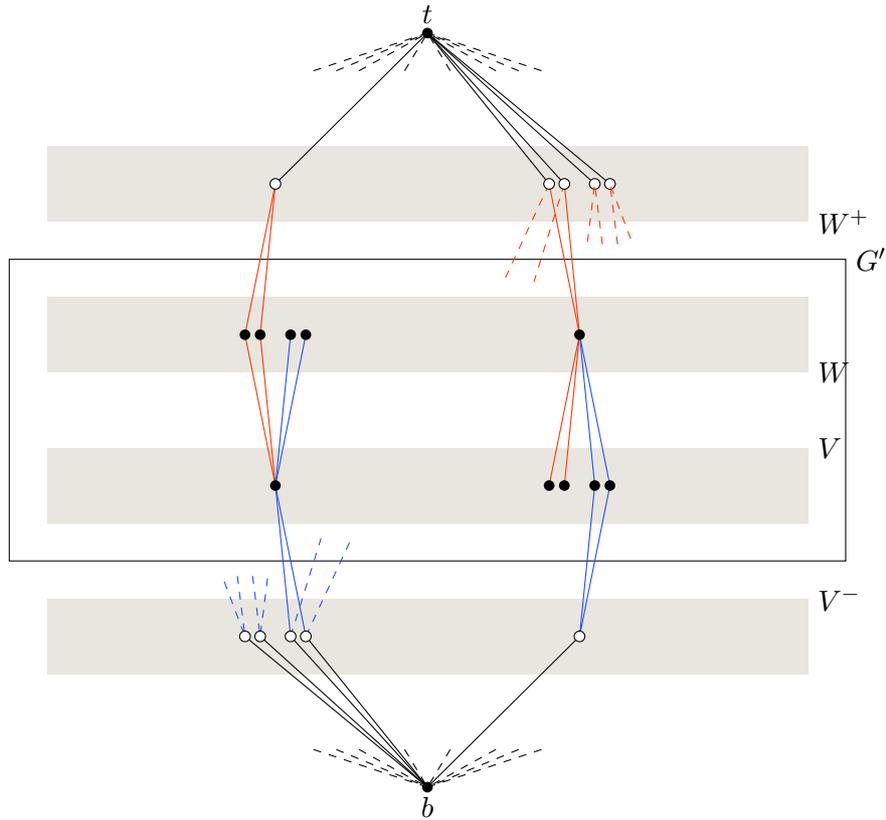
\begin{figure}[ht]
 \begin{center}
\begin{tikzpicture}
    \definecolor{red}{rgb}{1,0.2,0}
    \definecolor{blue}{rgb}{0.1,0.3,1}
    \definecolor{lightgray}{rgb}{0.91,0.9,0.87}
    \fill (0,5) circle [radius=2pt] node [above] {$t$};
    \fill (0,-5) circle [radius=2pt] node [below] {$b$};
    \foreach \x in {0.3, 0.9, 1.2, 1.5} {
        \draw[dashed] (\x, 4.5) -- (0, 5);
        \draw[dashed] (-\x, 4.5) -- (0, 5);
        \draw[dashed] (\x, -4.5) -- (0, -5);
        \draw[dashed] (-\x, -4.5) -- (0, -5);
    }
    \fill[lightgray] (-5, -3.5) rectangle (5, -2.5) node [black, right] {$V^-$};
    \fill[lightgray] (-5, -1.5) rectangle (5, -0.5) node [black, right] {$V$};
    \fill[lightgray] (-5, 1.5) rectangle (5, 0.5) node [black, right] {$W$};
    \fill[lightgray] (-5, 3.5) rectangle (5, 2.5) node [black, right] {$W^+$};

    \draw (-5.5, -2) rectangle (5.5, 2) node [right] {$G'$};

    \draw[red] (-2, -1) -- (-2.2, 1) -- (-2, 3);
    \draw[red] (-2, -1) -- (-2.4, 1) -- (-2, 3);
    \draw[blue] (-1.8, -3) -- (-2, -1) -- (-1.8, 1);
    \draw[blue] (-1.6, -3) -- (-2, -1) -- (-1.6, 1);
    \draw[blue, dashed] (-1.8, -3) -- (-1.4, -1.7);
    \draw[blue, dashed] (-1.6, -3) -- (-1.0, -1.7);
    \fill (-2, -1) circle [radius=2pt];
    \draw(-2, 3) -- (0, 5);
    \filldraw[fill=white] (-2, 3) circle [radius=2pt];
    \fill (-2.2, 1) circle [radius=2pt];
    \fill (-2.4, 1) circle [radius=2pt];
    \fill (-1.8, 1) circle [radius=2pt];
    \fill (-1.6, 1) circle [radius=2pt];
    \draw(-1.8, -3) -- (0, -5);
    \draw(-1.6, -3) -- (0, -5);
    \draw(-2.2, -3) -- (0, -5);
    \draw(-2.4, -3) -- (0, -5);
    \filldraw[fill=white] (-1.8, -3) circle [radius=2pt];
    \filldraw[fill=white] (-1.6, -3) circle [radius=2pt];
    \draw[blue, dashed] (-2.5, -2.2) -- (-2.4, -3) -- (-2.7, -2.2);
    \draw[blue, dashed] (-2.3, -2.2) -- (-2.2, -3) -- (-2.1, -2.2);
    \filldraw[fill=white] (-2.4, -3) circle [radius=2pt];
    \filldraw[fill=white] (-2.2, -3) circle [radius=2pt];

    \draw[blue] (2, 1) -- (2.2, -1) -- (2, -3);
    \draw[blue] (2, 1) -- (2.4, -1) -- (2, -3);
    \draw[red] (1.8, 3) -- (2, 1) -- (1.8, -1);
    \draw[red] (1.6, 3) -- (2, 1) -- (1.6, -1);
    \draw[red, dashed] (1.8, 3) -- (1.4, 1.7);
    \draw[red, dashed] (1.6, 3) -- (1.0, 1.7);
    \fill (2, 1) circle [radius=2pt];
    \draw(2, -3) -- (0, -5);
    \filldraw[fill=white] (2, -3) circle [radius=2pt];
    \fill (2.2, -1) circle [radius=2pt];
    \fill (2.4, -1) circle [radius=2pt];
    \fill (1.8, -1) circle [radius=2pt];
    \fill (1.6, -1) circle [radius=2pt];
    \draw(1.8, 3) -- (0, 5);
    \draw(1.6, 3) -- (0, 5);
    \draw(2.2, 3) -- (0, 5);
    \draw(2.4, 3) -- (0, 5);
    \filldraw[fill=white] (1.8, 3) circle [radius=2pt];
    \filldraw[fill=white] (1.6, 3) circle [radius=2pt];
    \filldraw[fill=white] (2.4, 3) circle [radius=2pt];
    \filldraw[fill=white] (2.2, 3) circle [radius=2pt];
    \draw[red, dashed] (2.5, 2.2) -- (2.4, 3) -- (2.7, 2.2);
    \draw[red, dashed] (2.3, 2.2) -- (2.2, 3) -- (2.1, 2.2);
\end{tikzpicture}
\end{center}
\caption{The Hasse diagram of $L$. The points in $P$ are denote by black dots, the points outside of $P$ are white.}
\end{figure}

We claim that $(L,\le)$ is a lattice, that $P=V\cup W\cup \{t,b\} \subseteq L$ is a woven subset of $L$ and that $P-p$ is not woven in $L$ for any $p\in P$. To show that $L$ is a lattice and that $P$ is woven in $L$ we have to show that there is, for every pair $x,y\in L$, a supremum and an infimum and that at least one of these two lies in $P$ if $x,y\in P$. We do so via a series of lemmas which distinguish different cases for $x,y$.

Let us first consider the case that either both $x$ and $y$ lie in $V$, or that they both lie in~$W$:
\begin{LEM}\label{lem:same_class}
 If $v_1,v_2\in V$, then there is a supremum and an infimum of $v_1,v_2$ in $L$. Moreover, if $v_1\meet v_2\neq b$ then $v_1\join v_2\in W$.
 
 Analogously, if $w_1,w_2\in W$, then there is a supremum and an infimum of $w_1,w_2$ in $L$. Moreover, if $w_1\join w_2\neq t$ then $w_1\meet w_2\in V$.
\end{LEM}
\begin{proof}
Let us start by showing that there is a supremum of $v_1$ and $v_2$.

First consider the case that the neighbourhoods of $v_1$ and $v_2$ in $G'$ intersect, that is, $N_{G'}(v_1)\cap N_{G'}(v_2)\neq \emptyset$. In this case, there is only one point in the intersection, since if there are $w_1,w_2\in N_{G'}(v_1)\cap N_{G'}(v_2), w_1\neq w_2,$ then $v_1w_1v_2w_2v_1$ would be a cycle of length $4$ in $G'$, contradicting the fact that $G'$ has girth at least $12$. We claim that the single point in the intersection, which we call $w$, is the supremum of $v_1$ and $v_2$.

To see this consider any $x\in L$ such that $v_1\le x, v_2\le x$. We need to show that $w\le x$. If $x=t$ then this is clear and $x\in W\cup V \cup V^- \cup \{b\}$ is not possible, so suppose that $x\in W^+$. Let $w_1,w_2$ be the neighbours in $W$ of $x$, i.e., $w_1,w_2\le x$. We show that $w_1=w$ or $w_2=w$. So suppose that $w\neq w_1,w_2$. Let $v_x\in V$ be the point corresponding to $x$. As $v_1\le x$ we may suppose without loss of generality that $v_1\le w_1$. Now if $v_2\le w_2$ then $wv_1w_1v_sw_2v_2w$ contains a cycle of length at most $6$ in $G'$ by \cref{lem:walk_cycle}, as $v_1\neq v_2$ and $w\not\in \{v_1,w_1,v_s,w_2,v_2\}$. This contradicts the fact that $G'$ has girth at least $12$. Thus $v_2\le w_1$ and hence $w_1=w$ as $N_{G'}(v_1)\cap N_{G'}(v_2)=\{w\}$, contradicting the assumption that $w\neq w_1$ and thus proving $w\le x$.

Now suppose that $N_{G'}(v_1)\cap N_{G'}(v_2)=\emptyset$.

Then every candidate for a supremum of $v_1$ and $v_2$ is either $t$, or lies in $W^+$, hence it is enough to show that there cannot be two elements $w^+_1,w^+_2\in W^+$ both satisfying $v_1,v_2\le w^+_1,w^+_2$. So suppose that there are two such points and denote the neighbours of $w_1^+$ and $w_2^+$ in $W$ as $w_{11},w_{12}$ and $w_{21},w_{22}$ respectively, i.e., $w_{11},w_{12}\le w^+_1$ and $w_{21},w_{22}\le w^+_2$.

As $v_1\le w^+_1,w^+_2$, we may suppose without loss of generality that $v_1\le w_{11},w_{21}$. Since $N_{G'}(v_1)\cap N_{G'}(v_2)=\emptyset$, we thus have $v_2\le w_{12},w_{22}$ and $w_{12},w_{22}\not\in \{w_{11},w_{21}\}$. Let us denote the corresponding points of $w^+_1$ and $w^+_2$ in $V$ as $v_{w^+_1}$ and $v_{w^+_2}$. Since $w^+_1\neq w^+_2$ either $w_{12}\neq w_{22}$ or $w_{11}\neq w_{21}$, as otherwise $G'$ would contain a cycle of length 4.
In any case, we consider the closed walk $v_1w_{11}v_{w^+_1}w_{12}v_2w_{22}v_{w^+_2}w_{21}v_1$.
As $v_{w^+_1}\neq v_{w^+_2}$, we have $v_1\neq v_{w^+_1}$ or $v_1\neq v_{w^+_2}$ and $v_2\neq v_{w^+_1}$ or $v_2\neq v_{w^+_2}$. Furthermore, either $w_{11}\notin \{w_{12},w_{21},w_{22}\}$ and $w_{21}\notin \{w_{11},w_{12},w_{22}\}$ or $w_{12}\notin \{w_{11},w_{21},w_{22}\}$ and $w_{22}\notin \{w_{11},w_{12},w_{21}\}$
This allows the application of \cref{lem:walk_cycle} to our walk, yielding a cycle of length at most 8, which contradicts the fact that $G'$ has girth at least $12$. Thus there exists a supremum $v_1\join v_2$ in $L$.

One candidate for the infimum $v_1\meet v_2$ is $b$. Every other candidate needs to lie in $V^-$. However, there can be at most one such candidate in $V^-$, otherwise, these candidates together with $v_1,v_2$ would correspond to a cycle of length $4$ in $G'$ contradicting the fact that $G'$ has girth at least $12$. Thus there is indeed an infimum $v_1 \meet v_2$.
 
 Moreover, if $v_1\meet v_2\neq b$, then there is a point $w\in W$ such that both, $v_1w$ and $v_2w$ are red edges in $G'$, hence $N_{G'}(v_1)\cap N_{G'}(v_2)\neq\emptyset$, which shows the moreover part of the claim.
 
 The statement for $w_1,w_2\in W$ follows by a symmetric argument.
\end{proof}
We can now apply \cref{lem:same_class} to show the existence of suprema and infima between $v\in V$ and $w\in W$:
\begin{LEM}\label{lem:different_class}
 If $v\in V$ and $w\in W$, then there is a supremum and an infimum of $v$ and $w$ in $L$. Moreover, if $v\meet w\neq b$ then $v\join w=t$ or $v\le w$.
\end{LEM}
\begin{proof}
If $v\le w$ then the statement is obvious, so suppose that $v\not\le w$.

By \cref{lem:same_class}, every point $w_i\in N_{G'}(v)$ has a supremum with $w$ which is either $t$ or lies in $W^+$. Moreover, there can be at most one point $w_i\in N_{G'}(v)$ such that the supremum $w_i\join w$ is in $W^+$, since if there are two, $w_1,w_2\in N_{G'}(v)$ say, then, by \cref{lem:same_class}, $w_1\meet w$ and $w_2\meet w\in$ both lie in $V$ and thus $wv_1w_1vw_2v_2w$ is a cycle of length $6$ in $G'$. 
Hence $v\join w$ is well-defined.

A symmetric argument shows that also $v\meet w$ is well-defined, so all that is left to show is that $v\join w\in W^+$ and $v\meet w\in V^-$ cannot both occur.

However, if this were the case, say $w^+=v\join w\in W^+$ and $v^-=v\meet w\in V^-$, we can consider the corresponding vertex $v_{w^+}$ of $w^+$ in $V$ and the  corresponding vertex $w_{v^-}$ of $v^-$ in $W$. By definition, there is a vertex $w_1\in W$ such that $vw_1\in E(G')$ and both $w_1v_{w^+}$ and $wv_{w^+}$ are blue edges. Similarly, there is a vertex $v_1\in V$ such that $v_1w\in E(G')$ and both $v_1w_{v^-}$ and $vw_{v^-}$ are red edges.
Consider the closed walk $vw_1v_{w^+}wv_1w_{v^-}v$.
We have $v\notin \{v_{w^+}, v_1\}$ as $v\not\le w$ and similarly $w\notin \{w_1, w_{v^-}\}$. Moreover, since every edge in $G'$ has precisely one colour we have $v_1w_{v^-}\neq v_{w^+}w_1$ and thus either $w_{v^-}\neq w_1$ or $v_1\neq v_{w^+}$.
We can thus apply \cref{lem:walk_cycle} to our walk to show the existence of a cycle of length at most $6$ in $G'$, which is a contradiction.
\end{proof}
Finally it remains to consider suprema $x \join y$ and infima $x \meet y$ where one of $x$ and $y$ lies in $V^-$ or $W^+$:
\begin{LEM}\label{lem:no_class}
 If $v^-\in V^-$ and $x\in L$, then there exists a supremum and an infimum of $v$ and $x$ in $L$.
 
 Similarly, if $w^+\in W^+$ and $x\in L$, then there exists a supremum and an infimum of $v$ and $x$ in $L$.
\end{LEM}
\begin{proof}
 If $v^-$ and $x$ are comparable, the statement is obvious, so suppose that this is not the case. It is then immediate that $v^-\meet x=b$.
 
 Let $v_1,v_2$ be the two points in $V$ such that $v^-=v_1\meet v_2$ and let $w_{v^-}$ be the point in $W$ corresponding to $v^-$. We note that any $l\in L$ satisfies $v^- < l$ if and only if $v_1\le l$ or $v_2\le l$. 
 We distinguish multiple cases, depending on whether $x$ lies in $W^+, W, V$ or $V^-$.

 If $x\in W^+$, then $x\join v^-=t$.
 
 If $x\in W$, then $x\join v_1$ and $x\join v_2$ exist by \cref{lem:different_class} and it is enough to show that $x\join v_1$ and $x\join v_2$ are comparable. If they are incomparable, then $x\join v_1\in W^+$ and $x\join v_2\in W^+$ and moreover $x\join v_1\neq x\join v_2$ and $v_1\not \le x\join v_2$ as well as $v_2\not \le x\join v_1$. Let $v_3\in V$ be the point corresponding to $x\join v_1$, let $v_4\in V$ be the point corresponding to $x\join v_2$, let $w_3\in W$ such that $w_3\join x=x\join v_1$ and let $w_4\in W$ such that $w_4\join x=x\join v_2$. Note that both $v_1,v_2,v_3,v_4$ and $x,w_{v^-},w_3,w_4$ consist of pairwise distinct points as $v_1\not\le x$ and $v_2\not \le x$ and $w_{v^-}\notin \{x,w_3,w_4\}$, thus $w_3v_3wv_4w_4v_2w_{v^-}v_1w_3$ needs to be a cycle of length $8$ in $G'$ contradicting the fact that $G'$ has girth at least $12$.
 
 If $x\in V$, then again $x\join v_1$ and $x\join v_2$ exist by \cref{lem:different_class}, and if they are incomparable we may suppose that $x\join v_1\in W^+\cup W$ and $x\join v_2\in W^+\cup W$ and moreover $x\join v_1\neq x\join v_2$. 
 
 If $x\join v_1\in W$ and $x\join v_2\in W$, then $v_1w_{v^-}v_2(x\join v_2)x(x\join v_1)v_1$ would be a cycle of length $6$ in $G'$ as $x\join v_1\neq x\join v_2$.
 
 Now suppose that $x\join v_1\in W$ and $x\join v_2\in W^+$. Let $v_{x_2}$ be the point in $V$ corresponding to $x\join v_2$ and let $w_1,w_2\in W$ such that $w_1\join w_2=x\join v_2$. We may suppose that $w_1,w_2\neq x\join v_1$ and that $v_2\le w_1$ and $x\le w_2$.  Note that $v_{x_2}\neq x$ as otherwise $x\join v_x=w_2$. Now $xw_2v_{x_2}w_1v_2w_{v^-}v_1(x\join v_1)x$ contains a cycle of length at most $8$ in $G'$ by \cref{lem:walk_cycle}, as $s\notin \{v_1,v_2,v_{x_2}\}$ and $x\join v_1\neq w_2$.
 
 So we may suppose that $x\join v_1\in W^+$ and $x\join v_2\in W^+$. 
 
 Let $v_{x_1}$ be the point in $V$ corresponding to $x\join v_1$, $v_{x_2}$ be the point in $V$ corresponding to $x\join v_2$, let $w_1,w_2,w_3,w_4\in W$ such that $w_1\join w_2=x\join v_1$ and $w_3\join w_4=x\join v_2$. We may suppose that $v_1\le w_1$, $v_2\le w_3$ and $x\le w_2,w_4$. Note that $x\notin \{v_1,v_2,v_{x_1},v_{x_2}\}$ and that $w_4\neq w_2$ as otherwise $w_4\le x\join v_1$ and thus $x\join v_2= x\join w_4\le x\join v_1$. Thus $xw_2v_{x_1}w_1v_1w_{v^-}v_2w_3v_{x_2}w_4x$ contains a cycle in $G'$ of length at most $10$ by \cref{lem:walk_cycle}.
 
 So the remaining case is $x\in V^-$. Let us denote the vertex in $W$ corresponding to $x$ as $w_x$ and let $v_3,v_4\in V$ such that $v_3\meet v_4=x$. Since every candidate for a supremum of $v^-$ and $x$ lies above one of $v_1\join v_3, v_1\join v_4, v_2\join v_3$ and $v_2\join v_4$, all of which exist by \cref{lem:different_class}, it is enough to show that all these points are comparable, since then the smallest of them needs to be the supremum of $v^-$ and $x$.
 
 However, we know by the previous argument that $v^-\join v_3$ exists, which needs to be equal to $v_1\join v_3$ or $v_2\join v_3$. Hence $v_1\join v_3$ and $v_2\join v_3$ are comparable.
 
 Similarly, if we consider $v^-\join v_4$ we see that $v_1\join v_4$ and $v_2\join v_4$ are comparable.
 
 If we consider $x\join v_1$, we observe that $v_1\join v_3$ and $v_1\join v_4$ are comparable.
 
 And finally, if we consider $x\join v_2$, we see that  $v_2\join v_3$ and $v_2\join v_4$ are comparable as well and therefore there indeed exists a supremum of $v^-$ and $x$.
\end{proof}
We have now seen that $L$ is indeed a lattice and that $P$ is woven in $L$. This allows us to state and prove the main result of this section:
\begin{THM}
 $L$ is a lattice and $P=V\cup W\cup \{t,b\}\subseteq L$ is woven in $L$ such that $P-p$ is not woven in $L$ for any $p\in P$.
\end{THM}
\begin{proof}
 By \cref{lem:same_class,lem:different_class,lem:no_class} $L$ is indeed a lattice. To see that $P$ is woven in $L$ observe that by \cref{lem:same_class}, \cref{lem:different_class} and the fact that $t$ and $b$ are comparable with every element in $P$ it follows that at most one of $x\join y$ and $x\meet y$ lie outside of $P$, for any $x,y\in P$.

 For any $p\in V$ there are $w_1,w_2\in W$ such that $pw_1$ and $pw_2$ are both blue edges in $G'$, thus both $w_1\join w_2$ and $w_1\meet w_2$ lie outside of $P-p$. Similarly, $P-p$ is not woven in $L$ for any $p\in W$. Finally, if $p=b$ we note that there are $v_1,v_2\in V$ such that $v_1\join v_2\in W^+$ which implies that $v_1\meet v_2=b$ and shows that $P-b$ is not woven in $L$. Similarly, $P-t$ is not woven in $L$.
\end{proof}

As before, this result about woven subsets of lattices allows us to directly obtain a result about structurally submodular separation systems, as we can use this lattice $L$ to construct a universe $\vU$ of separations together with a structurally submodular separation system $\vS\subseteq \vU$ which cannot be unravelled:
\begin{THM}\label{thm:universe}
 There exists a universe $\vU$ of separations and a submodular subsystem $\vS\subseteq \vU$ such that $\vS-\{\vs,\sv\}$ is not submodular in $\vU$ for any $\vs\in \vS$.
\end{THM}
\begin{proof}
    Let $L'$ be a copy of $L$ with reversed partial order, i.e., the poset-dual of $L$.
    In the disjoint union $L \sqcup L'$ we now identify the copy of $t$ in $L$ (the top of $L$) with the copy of $b$ in $L'$ (the top of $L'$) and the copy of $b$ in $L$ with the copy of $t$ in $L'$ to obtain $\vU$.
    It is easy to see that this forms a universe of separations and that $\vS=P\cup P'$ (where $P\subseteq L$ is as above and $P'\subseteq L'$ is the image of $P$ in $L'$) is a separation system which is submodular in $\vU$. Moreover, there is no separation $\vs\in \vS$ such that $\vS-\{\vs,\sv\}$ is again submodular in $\vU$.
\end{proof}

Note that neither our lattice $L$ nor the constructed universe $\vU$ of separations are distributive.

\section{Woven posets}
\label{sec:weak_unravel}

Instead of asking in \cref{prop:woven_lattice} for a woven subset $P$ inside a lattice $L$, we might as well directly ask for a partially order set $P$, which is woven in itself. More precisely let us say that a partially order set $P$ is \emph{woven} if we have, for any two elements $p,q$ of $P$ a supremum or an infimum \emph{in $P$}, i.e., there exists a $r\in P$ such that $p\le r,q\le r$ and $r\le s$ whenever $s\in P$ such that $q\le s$ and $p\le s$ or there exists a $r\in P$ such that $p\ge r,q\ge r$ and $r\ge s$ whenever $s\in P$ such that $q\ge s$ and $p\ge s$.

The \emph{Dedekind-MacNeille-completion}~\cite{MacNeille} from lattice theory implies that we can find, for each poset $P$, a lattice $L$ in which $P$ can be embededded in such a way that existing joins and meets in $P$ are preserved. Hence if $P$ is a finite woven set there exists a lattice $L$ in which $P$ can be embedded so that the image of $P$ in $L$ is woven in $L$.

Using this notion of wovenness inside the poset itself, we can now weaken the concept of unravelling, by considering a woven poset $P$ instead of a woven subset of a lattice. We will be able to show that, given a woven poset $P$, we can always remove a point so that the remainder is again a woven poset.

Even though every woven poset can be embedded into a lattice, this still is a proper weakening of the unravelling conjecture.
The key difference here lies in the different perspective we take on $P-p$, given a poset $P$ and some $p \in P$:
if we consider $P$ as a woven poset and $P-p$ is again woven, then there are lattices $L$ and $L'$ in which $P$ and $P-p$, respectively, can be embedded so that the images are woven as subset of these lattice. However, these two lattices are different, and in general it is not possible to find one lattice in which both $P$ and $P-p$ can be embedded so that their images are woven in that lattice. In this sense, having an unravelling for the wovenness of a poset is a weaker property than having an unravelling as a woven subset of a lattice.

To prove this weaker unravelling property for woven posets we will show that every woven poset contains a point $p$ with precisely one lower (or one upper) cover, i.e. there exists precisely one $q$ such that  $p>q$ ($p<q$) and there does not exists any $c\in P$ such that $p>c>q$ ($p<c<q$). Deleting such a point does not destroy the wovenness, as shown by the following lemma:

\begin{LEM}\label{lem:delete_deg1}
 Let $P$ be a woven poset and $p\in P$ a point with precisely one lower (upper) cover $p'$,
 then $P' = P - p$ is a woven poset.
\end{LEM}
\begin{proof}
 Let $x,y \in P'$. We need to show that $x,y$ have a supremum or an infimum in $P'$.
 If they have a supremum $s$ in $P$, then $s \neq p$: as $p'$ is the only lower cover of $p$ we have $x,y\le p'$ as soon as $x,y\le p$. Thus $s \in P'$ is also the supremum of $x$ and $y$ in $P'$.
 
 If $x,y$ have an infimum $z$ in $P$, then either $z \neq p$ and $z$ is also the infimum in $P'$ or
 $z = p$, in which case $p'$ is the infimum of $x$ and $y$ in $P'$, as $p'$ is the only lower cover of $p$.
 
 The upper cover case is dual.
\end{proof}
Thus, what is left to show is that there always exists a point $p\in P$ with precisely one upper or precisely one lower cover. To see this, we consider the maximal elements of $P$, since any subset of them needs to have an infimum by the following lemma:
\begin{LEM}\label{lem:infimum}
 Let $P$ be a woven poset and $M$ its set of maximal elements.
 Then every non-empty subset $M' \sub M$ has an infimum $\inf M'$ in $P$.
\end{LEM}
\begin{proof}
 We proceed by induction on $\abs{M'}$.
 For the induction start $\abs{M'} = 1$ this is trivial.
 For the induction step consider $\abs{M'} \geq 2$ and let $m \in M'$ and $M'' \coloneqq M' - m$.
 By the inductive hypothesis $M''$ has an infimum $p$.
 Since $m$ is maximal there can only be a supremum of $m$ and $p$ if $m$ and $p$ are comparable. However then there also exists an infimum of $m$ and $p$ in $P$. Thus, as $P$ is woven, in any case $P$ needs to contain an infimum $q$ of $m$ and $p$.
 This $q$ lies below all of $M'$ and, conversely, every point which lies below all of $M'$ lies below both $p$ and $m$ and hence below $q$.
 Thus $q$ is the infimum of $M'$ in $P$.
\end{proof}
Given a woven poset $P$, let $M$ be the set of maximal elements of $P$. Given some subset $M'\subseteq M$ we are interested in those points $x\in P$ where, for every maximal element $m\in M$ we have $x\le m$ precisely if $m\in M'$. Let us denote as $d(M')$ the set of all these points in $P$.

Either each such set $d(M')$ just consist of at most one point, or there is some $M'$ such that $d(M')$ has size more then one. In the latter case, the following lemma guarantees that we find a point $p\in P$ with only one upper cover: 
\begin{LEM}\label{lem:max_upper}
 Let $P$ be a woven poset and $M$ the set of maximal elements of $P$. If $M'\subseteq M$ is subset-minimal with the property that $d(M')$ contains at least two points, then there is an $x\in d(M')$ for which $\inf M'$ is the only upper cover.
\end{LEM}
\begin{proof}
 Observe that, if $d(M)\neq \emptyset$ then $\inf M'\in d(M)$. Let $x$ be a maximal element of $d(M')-\inf M'$. Since $x$ is a candidate for $\inf M'$, we have that $\inf M'$ is an upper cover of $x$. If $y$ is any point other than $\inf M'$ such that $x<y$ then $y$ lie in $d(M'')$ for some proper subset $M''$ of $M$. Thus, by our assumption, $y$ is the only element of $d(M'')$ and therefore $y=\inf M''$. However, $\inf M'\le \inf M''$ and $y\neq \inf M'$, thus $y$ is not an upper cover of $x$.
\end{proof}
It remains to consider the case where every $d(M')$ has size one. However, in that case we can find an element with only one lower cover, as shown in the following lemma:
\begin{LEM}\label{lem:unique_cover}
 Let $P$ be a woven poset.
 Then $P$ has an element which has precisely one lower or one upper cover.
\end{LEM}
\begin{proof}
 Suppose the converse is true. Let $M$ be the set of maximal elements of $P$. 
 Note that every element of $P$ lies in $d(M')$ for exactly one set $M'\subseteq M$. 
 By \cref{lem:max_upper}, given any $M'\subseteq M$ there exists at most one element in $d(M')$. Moreover, by \cref{lem:infimum} we know that $\inf M'$ exists for every $M'\subseteq M$.
 
 Now if $\abs{d(M')}=1$ for some $M'\subseteq M$, then $\inf M'\in d(M')$: we know that $\inf M'$ is in $d(M'')$ for some $M''\subseteq M$ and clearly $M'\subseteq M''$, however if $ d(M')=\{v\}$, say, then clearly $v\le \inf M'$ which implies that $M''\subseteq M'$ and thus $M'=M''$. 
 
 However, since every element of $P$ lies in some $d(M')$ and $\inf M'\le \inf M''$ whenever $M''\subseteq M'$ this implies that $\inf M$ is the smallest element of $P$. However, any upper cover of this smallest element $\inf M$ has precisely one lower cover, which is a contradiction.
\end{proof}

Thus if we consider woven posets instead of woven subsets of a fixed lattice (as in \cref{sec:sandwich}) we can indeed unravel every such poset: given some woven poset $P$, by \cref{lem:unique_cover}, $P$ contains an element $p$ which has only one upper or lower cover, and, by \cref{lem:delete_deg1}, $P-p$ is again woven. Thus we obtain the following theorem:
\thmselfwoven*

Again we can translate this result to abstract separation systems as introduced in \cref{sec:history}.

Let us say that a separation system $\vS$, on its own, not in the context of a surrounding universe $\vU$ of separations, is \emph{submodular} if there exists, for any two separations $\vs,\vt\in \vS$ a supremum or an infimum in $\vS$, i.e., -- as for woven posets -- we require that there either is a smallest separation $\vr$ such that $\vs,\vt\le \vr$ or there is a largest separation $\vr$ such that $\vs,\vt\ge\vr$. These submodular separation systems are also considered in \cite{Structuresubmodular}, where we also show that one can find, for each such system $\vS$, a universe $\vU$ of separations in which we can embed $\vS$ so that the joins and meets in $\vS$ are preserved.

We now obtain the following corollary for this type of separation system:
\begin{THM}
 Let $\vS$ be a submodular separation system. Then there exists an $\vs\in \vS$ such that $\vS\sm\{\vs,\sv\}$ is again submodular.
\end{THM}
\begin{proof}
Observe that $\vS$ considered as a poset is woven.
 Let $M$ be the set of maximal elements of $\vS$. We note that $\vs\ge \tv$ for all $\vs,\vt\in M$. Therefore $\inf M\ge \tv$ for all $\vt\in M$ and thus $\inf M\ge \sup M^*=(\inf M)^*$. Suppose that there is a proper subset $M'$ of $M$ such that $\abs{d(M')}\ge 2$ and let $M'$ be chosen subset-minimal with that property. Let $\vx\in d(M')$ be as guaranteed by \cref{lem:max_upper}. We note that $\vx\neq (\inf M')^*$ as otherwise $\vx\le (\inf M)^*\le \inf M$, contradicting the fact that $\vx\in d(M')$. But this implies that $\vS-\vx$ is a woven poset by \cref{lem:delete_deg1}. However, $\xv$ has only one lower cover in $\vS$ and, since this cover is not $\vx$, also exactly one lower cover in $\vS-\vx$. Thus, again by \cref{lem:delete_deg1}, also $((\vS-\vx)-\xv)$ is a woven poset and thus $S-x$ is a submodular separation system.
 
 Hence we may suppose that $\abs{d(M')}\le 1$ for all proper subset $M'$ of $M$. This implies that every element $\vs\in \vS$ is nested with $\inf M$: if $\vs\in d(M)$ then $\vs\le \inf M$ and if $\vs\in d(M')$ for a proper subset $M'$ of $M$, then $\vs=\inf M'\ge \inf M$. Now suppose that $\abs{M}\ge 2$. Then there is a $\vm\in M$ such that $\mv\neq \inf M$. We claim that $\vS\sm \{\vm,\mv\}$ is again submodular. To see this suppose that, for some $\vx,\vy\in \vS$, we have that $\vx\join \vy=\vm$ (the case $\vx\meet\vy=\mv$ is dual). As $\vx$ and $\vy$ are nested with $\inf M$ this implies that $\vx,\vy\ge \inf M$ as $\vx\le \inf M$ would imply that $\vx\join \vy=\vy$ or $\vx\join \vy\le \inf M$. Thus $\vx=\inf M'$ and $\vy=\inf M''$ for subsets $M',M''$ of $M$, say. Thus $\inf (M'\cup M'')$, which exists by \cref{lem:infimum},
 is also the infimum of $\vx$ and $\vy$. Moreover, since $\mv\neq \inf M$ and $\mv$ is a minimal element of $\vS$ and $\inf (M'\cup M'')\ge \inf M$ we have that $\inf (M'\cup M'')\neq \mv$ and thus there is a corner of $\vx$ and $\vy$ in $\vS\sm \{\vm,\mv\}$.
 
 It remains the case that $\abs{M}=1$, say $M=\{\vm\}$. In this case however, we have that $\vs\le \vm$ for every $\vs\in \vS$. If $\vS=\{\vm,\mv\}$ the statement is trivial, so let $\vs\in \vS-\vm$ be $\le$-maximal such that $\vs\neq\mv$. Such an $\vs$ exists as $\mv$ is a $\le$-minimal element of $\vS$. Then $\vm$ is the unique upper-cover of $\vs$. Thus $\vS-\vs$ is a woven poset by \cref{lem:delete_deg1}. Moreover, $\mv$ is the unique lower cover of $\sv$ and, since $\mv\neq \vs$ it is also the unique lower cover of $\sv$ in $\vS-\vs$. Thus $(\vS-\vs)-\sv$ is a woven poset by \cref{lem:delete_deg1}, and thus $S-s$ is a submodular separation system.
\end{proof}

The Dedekind-MacNeille completion of posets \cite{LatticeBook} allows us to embed every woven poset into a lattice so that the poset is woven in this lattice.
We show in \cite{Structuresubmodular} that this technique can also be applied to submodular separation systems to obtain a universe of separations in which the separation system is submodular.

In particular, if $P$ is a woven poset and $p\in P$ such that $P'=P-p$ is again woven, there are lattices $L$ and $L'$ such that $P$ is woven in $L$ and $P'$ is woven in $L'$.
If we could arrange for these two lattices to be sublattices of one another, $L' \subseteq L$, in such a way that  every element of $P' \subseteq L'$ is mapped to the corresponding element of $P \subseteq L$, then this would imply that $P$ could be unravelled as a woven subset of $L$ in the sense of \cref{prop:woven_lattice}.

The way in which we constructed $P'$, however, makes this almost impossible.
We choose $p$ as an element with a unique upper, or a unique lower cover.
Now if $p\in P$ has a unique upper cover $q$, say, and is also the supremum of some two points $r,s\in P \sm \{p\}$, then the Dedekind-MacNeille completion $L'$ of $P'$  cannot be embedded in the way outlined above into the Dedekind-MacNeille completion $L$ of $P$: in $L'$, the images of $r$ and $s$ have the image of $q$ as supremum and an embedding as a sublattice would need to preserve this property, but the images of $r$ and $s$ in $L$ have the image of $p$ as their supremum. (However, $L'$ \emph{is} order-isomorphic to a subposet of $L$.)

\bibliography{JMC}

\begin{bibdiv}
\begin{biblist}

\bib{LatticeBook}{book}{
      author={Davey, B.~A.},
      author={Priestley, H.~A.},
       title={Introduction to {L}attices and {O}rder},
     edition={2},
   publisher={Cambridge University Press},
        date={2002},
         DOI={10.1017/CBO9780511809088.009},
}

\bib{DiestelBook16noEE}{book}{
      author={Diestel, Reinhard},
       title={{Graph Theory}},
     edition={5},
      series={Graduate Texts in Mathematics},
   publisher={Springer},
        date={2017},
      volume={173},
        note={\doi{10.1007/978-3-662-53622-3}},
}

\bib{AbstractSepSys}{article}{
      author={Diestel, Reinhard},
       title={Abstract separation systems},
        date={2018},
     journal={Order},
      volume={35},
       pages={157\ndash 170},
         DOI={10.1007/s11083-017-9424-5},
}

\bib{ProfileDuality}{article}{
      author={Diestel, Reinhard},
      author={Eberenz, Ph.},
      author={Erde, J.},
       title={Duality theorem for blocks and tangles in graphs},
        date={2017},
     journal={SIAM J.\ Discrete Math.},
      volume={31},
      number={3},
       pages={1514\ndash 1528},
         DOI={10.1137/16M1077763},
}

\bib{AbstractTangles}{article}{
      author={Diestel, Reinhard},
      author={Erde, J.},
      author={Wei{\ss}auer, D.},
       title={Structural submodularity and tangles in abstract separation
  systems},
        date={2019},
     journal={J. Combin. Theory (Series~A)},
      volume={167C},
       pages={155\ndash 180},
        note={\doi{10.1016/j.jcta.2019.05.001}},
}

\bib{ProfilesNew}{article}{
      author={Diestel, Reinhard},
      author={Hundertmark, F.},
      author={Lemanczyk, S.},
       title={Profiles of separations: in graphs, matroids, and beyond},
        date={2019},
     journal={Combinatorica},
      volume={39},
      number={1},
       pages={37\ndash 75},
         DOI={10.1007/s00493-017-3595-y},
}

\bib{TangleTreeGraphsMatroids}{article}{
      author={Diestel, Reinhard},
      author={Oum, S.},
       title={Tangle-tree duality in graphs, matroids and beyond},
        date={2019},
     journal={Combinatorica},
      volume={39},
       pages={879\ndash 910},
         DOI={10.1007/s00493-019-3798-5},
}

\bib{TangleTreeAbstract}{article}{
      author={Diestel, Reinhard},
      author={Oum, S.},
       title={Tangle-tree duality in abstract separation systems},
        date={2021},
     journal={Advances in Mathematics},
      volume={377},
       pages={107470},
         DOI={10.1016/j.aim.2020.107470},
}

\bib{CanonicalToT}{article}{
      author={Elbracht, {Chr}istian},
      author={Kneip, Jakob},
       title={A canonical tree-of-tangles theorem for submodular separation
  systems},
        date={2020},
      eprint={2009.02091},
}

\bib{Structuresubmodular}{unpublished}{
      author={Elbracht, {Chr}istian},
      author={Kneip, Jakob},
      author={Teegen, Maximilian},
       title={The structure of submodular separation~systems},
        note={In preparation},
}

\bib{ToTviaTTD}{article}{
      author={Elbracht, {Chr}istian},
      author={Kneip, Jakob},
      author={Teegen, Maximilian},
       title={Obtaining trees of tangles from tangle-tree duality},
        date={2020},
      eprint={2011.09758},
        note={To appear in \emph{Journal of Combinatorics}},
}

\bib{FiniteSplinters}{article}{
      author={Elbracht, {Chr}istian},
      author={Kneip, Jakob},
      author={Teegen, Maximilian},
       title={Trees of tangles in abstract separation systems},
        date={2021},
     journal={J. Combin. Theory (Series~A)},
      volume={180},
       pages={105425},
         DOI={10.1016/j.jcta.2021.105425},
}

\bib{TanglesInMatroids}{article}{
      author={Geelen, J.},
      author={Gerards, B.},
      author={Whittle, G.},
       title={Tangles, tree-decompositions and grids in matroids},
        date={2009},
     journal={J.~Combin.\ Theory (Series B)},
      volume={99},
       pages={657\ndash 667},
         DOI={10.1016/j.jctb.2007.10.008},
}

\bib{kneip2020tangles}{misc}{
      author={Kneip, Jakob},
       title={Tangles and where to find them},
      status={Staats- und Universit{\"a}tsbibliothek Hamburg Carl von Ossietzky
  (2020)},
        note={Doctoral thesis. URN:
  \href{https://nbn-resolving.org/urn\%3Anbn\%3Ade\%3Agbv\%3A18-ediss-89005}{urn:nbn:de:gbv:18-ediss-89005}},
}

\bib{LAZEBNIK1995275}{article}{
      author={Lazebnik, Felix},
      author={Ustimenko, Vasiliy~A.},
       title={Explicit construction of graphs with an arbitrary large girth and
  of large size},
        date={1995},
        ISSN={0166-218X},
     journal={Discrete Applied Mathematics},
      volume={60},
      number={1},
       pages={275 \ndash  284},
         DOI={10.1016/0166-218X(94)00058-L},
  url={http://www.sciencedirect.com/science/article/pii/0166218X9400058L},
}

\bib{MacNeille}{article}{
      author={MacNeille, H.~M.},
       title={Partially ordered sets},
        date={1937},
        ISSN={0002-9947},
     journal={Trans. Amer. Math. Soc.},
      volume={42},
      number={3},
       pages={416\ndash 460},
         DOI={10.2307/1989739},
         url={https://doi.org/10.2307/1989739},
}

\bib{GMX}{article}{
      author={Robertson, Neil},
      author={Seymour, P.D},
       title={Graph minors. {X}. {Obstructions} to tree-decomposition},
        date={1991},
        ISSN={0095-8956},
     journal={J.~Combin.\ Theory (Series B)},
      volume={52},
      number={2},
       pages={153\ndash 190},
         DOI={10.1016/0095-8956(91)90061-N},
  url={http://www.sciencedirect.com/science/article/pii/009589569190061N},
}

\end{biblist}
\end{bibdiv}

\medskip
\noindent
\begin{minipage}{\linewidth}
 \raggedright\small
   \textbf{Christian Elbracht},
   \texttt{christian.elbracht@uni-hamburg.de}

   \textbf{Jakob Kneip},
   \texttt{jakob.kneip@uni-hamburg.de}
   
   \textbf{Maximilian Teegen},
   \texttt{maximilian.teegen@uni-hamburg.de}

   Universität Hamburg,
   Hamburg, Germany
\end{minipage}
\end{document}